\documentclass[12pt,a4]{article}
\usepackage[top=30truemm,bottom=30truemm,left=27truemm,right=27truemm]{geometry}
\usepackage{lipsum} % for generating filler text
\usepackage{titlesec}
\titleformat{\subsection}[runin]% runin puts it in the same paragraph
       {\normalfont\bfseries}% formatting commands to apply to the whole heading
       {\thesubsection}% the label and number
       {0.5em}% space between label/number and subsection title
       {}% formatting commands applied just to subsection title
       [.]% punctuation or other commands following subsection title
\usepackage{amsmath,amsthm,amssymb,mathdots}
\usepackage{setspace,shuffle}
\usepackage{color}
\allowdisplaybreaks

\newtheorem{theorem}{Theorem}
\newtheorem{proposition}[theorem]{Proposition}

\newtheorem{corollary}[theorem]{Corollary}

\theoremstyle{remark}

\numberwithin{equation}{section}
\newcommand{\R}{\mathbb R}
\newcommand{\Z}{{\mathbb Z}}
\newcommand{\N}{{\mathbb N}}
\newcommand{\C}{{\mathbb C}}

\begin{document}
\begin{center} {\Large The harmonic strength of shells of lattices \\and weighted theta series} \\ Koji Tasaka \\ School of Information Science and Technology \\ Aichi Prefectural University \end{center}

%\maketitle

%%%%%%%%%%%%%%%%%%%%%%%%
\section{Main results of the talk}
In this talk, I presented two kinds of different results related to the $D_4$ lattice, obtained in a joint work \cite{HNT} with M.~Hirao and H.~Nozaki.
Let me first overview the results.

Our interest is in a design-theoretic property of the $2m$-{\itshape shell} 
\[ (D_4)_{2m}=\{ (x_1,\ldots,x_4)\in \Z^4 \mid x_1^2+\cdots+x_4^2=2m\}\]
of {\itshape the $D_4$ lattice} (see \S\ref{sec:Dn} for the definition).
We first proved that for any $m\in \N$, the subset $\frac{1}{\sqrt{2m}}(D_4)_{2m}$ of the three-dimensional unit sphere $\mathbb{S}^3$ is an antipodal spherical $\{10,4,2\}$-design (the case $m=1$ was shown in \cite{Pache05}).
Here, for a subset $T\subset \N$, a finite subset $X$ of the unit sphere $\mathbb{S}^{d-1}$ of the $d$-dimensional Euclidean space $\R^d$ is a spherical design of harmonic index $T$ ({\itshape spherical $T$-design} for short) if the sum of all values $P(x)$ at $x\in X$ is zero for all harmonic polynomials $P(x)\in\R[x_1,\ldots,x_d] $ of degree $\ell\in T$.
%Note that $-x\in \frac{1}{\sqrt{2m}}(D_4)_{2m}$ holds for any $x\in\frac{1}{\sqrt{2m}}(D_4)_{2m}$, so every odd positive integer is 
%.
Remark that a spherical $\{t,t-1,\ldots,2,1\}$-design is called a {\itshape spherical $t$-design} due to Delsarte-Geothals-Seidel \cite{DelsarteGoethalsSeidel77}.
The concept of a spherical $T$-design was first introduced in \cite{DS89} and recently studied in \cite{BBXYZ,BannaiOkudaTagami15,OkudaYu16,ZhuBannaiBannaiKimYu17}.
As a convention, since lattices are abelian groups, every positive odd integer is contained in $T$, so is omitted to write.

A basic problem is to determine the maximal subset $T\subset \N$, {\itshape the harmonic strength}, such that $\frac{1}{\sqrt{2m}}(D_4)_{2m}$ is a spherical $T$-design.
This problem is deeply connected to the non-vanishing problem of the Fourier coefficients of cusp forms.
As a prototype, it was established by Venkov, stated in \cite[Proposition B]{Pache05}, that the $2m$-shell $\frac{1}{\sqrt{2m}}(E_8)_{2m} \subset \mathbb{S}^7$ of the $E_8$ lattice is a spherical 8-design if and only if Ramanujan's $\tau$-function $\tau(m)$ is zero, where $\sum_{m\ge1} \tau(m) q^m =\eta(z)^{24}\ (q=e^{2\pi i z})$ is the unique cusp form of weight 12 on ${\rm SL}_2(\Z)$ and $\eta(z)=q^{1/24}\prod_{n\ge1}(1-q^n)$ is the Dedekind eta function.
According to Lehmer's conjecture \cite{Lehmer47}, $\tau(m)$ is believed to be non-zero for all $m\in \N$, or equivalently, $\frac{1}{\sqrt{2m}}(E_8)_{2m}$ would never be a spherical 8-design.
In our work \cite{HNT}, we pointed out an analogy to this for the $D_4$ lattice. %, which may not be new to experts, but has not been written elsewhere.

\begin{theorem}\label{thm:main1}
For $m\in \N$, the harmonic strength of $\frac{1}{\sqrt{2m}}(D_4)_{2m}$ contains 6 if and only if the $m$th Fourier coefficient $\tau_2(m)$ of the unique newform $\eta(z)^8\eta(2z)^8=\sum_{m\ge1} \tau_2(m)q^m$ of weight 8 for $\Gamma_0(2)$ is zero.
\end{theorem}

The non-vanishing conjecture on the $\tau_2$-function is not formulated in the literature as far the author knows (see \cite{Cohen} for a similar project).
As quick numerical evidence by PARI-GP, one can easily check $\tau_2(m)\neq0$ up to $m\le 10^8$.
In \cite{HNT}, we obtained $\tau_2(p)\neq0$ for any prime $p$ which is not congruent to $-1$ modulo 15 (note that the minimum of $m\in \N$ such that $\tau_2(m)=0$ is prime).
We will discuss an extension of the above result to other positive even integers $\ell$. See \S3.

\

The next result is about a characterization (a classification) of spherical $\{10,4,2\}$-designs in $\mathbb{S}^3$.
For an {\itshape antipodal} subset $X \subset \mathbb{S}^{d-1}$ (i.e., $-x\in X$ for all $x\in X$), a subset $X'$ of $X$ is said to be a {\itshape half set} if $X$ is a disjoint union of $X'$ and $-X'$.
It follows that any half set $\frac{1}{\sqrt{2}}(D_4)_2'$ of $\frac{1}{\sqrt{2}}(D_4)_2$ (the 2-shell $(D_4)_2$ is known as the $D_4$ root system) is a spherical $\{10,4,2\}$-design with 12 points in $\mathbb{S}^3$.
We can also compute the inner product 
\[ A\left( \frac{1}{\sqrt{2}}(D_4)_2'\right) = \left\{-\frac12,0,\frac12\right\},\]
where for $X\subset \mathbb{S}^{d-1}$ we write $A(X)=\{\langle x,y\rangle \in [-1,1)\mid x,y\in X, x\neq y\}$ with the Euclidean inner product  $\langle x,y\rangle=x_1y_1+\cdots+x_dy_d \ (x,y\in \R^d)$.
These data can be used in the linear programming method to prove that, if a finite subset $Y\subset \mathbb{S}^3$ is a spherical $\{10,4,2\}$-design, then $|Y|\ge12$ (interestingly, our proof indicates that $Y$ attains the lower bound if and only if $A(Y)\subset\{-1/2,0,1/2\}$). 
%\item[(2)] If a finite subset $Y\subset \mathbb{S}^3$ satisfies $A(Y)\subset [-1/2,1/2]$, then $|Y|\le 12$.
%Moreover, $Y$ attains the upper bound if and only if $Y$ is a spherical $\{10,4,2\}$-design and $A(Y)\subset\{-1/2,0,1/2\}$.
This shows that the set $\frac{1}{\sqrt{2}}(D_4)_2$ is a `{\itshape tight}' antipodal spherical $\{10,4,2\}$-design. 
In our work, we classify such designs up to orthogonal transformations.
Here we note the fact\footnote{This is a consequence of the fact that the group $O(\R^d)$ acts on the space ${\rm Harm}_\ell(\R^d)$ of harmonic polynomials, i.e., $(\sigma^\ast P)(x):=P(\sigma^{-1}(x))$ is again a harmonic polynomial for any $P\in {\rm Harm}_\ell(\R^d)$ and $\sigma\in O(\R^d)$.} that if a finite subset $X\subset \mathbb{S}^{d-1}$ is a spherical $T$-design, then the set $\sigma(X)=\{\sigma(x)\mid x\in X\}$ is also a spherical $T$-design for any orthogonal transformation $\sigma\in O(\R^d) =\{\sigma:\R^d\rightarrow \R^d:\ \mbox{$\R$-linear map s.t.}\ \langle \sigma(x),\sigma(y)\rangle=\langle x,y\rangle, \ \forall x,y\in \R^d\}$. %, where $O(\R^d)$ denotes the group consisting of orthogonal transformations on $\R^d$.

\begin{theorem}\label{thm:main2}
Every antipodal spherical $\{10,4,2\}$-design with $24$ points on $\mathbb{S}^3$ is an orthogonal transformation of $\frac{1}{\sqrt{2}}(D_4)_2$.
\end{theorem}

Theorem \ref{thm:main2} is called a {\itshape uniqueness theorem} in the study of classification of spherical designs (see e.g., \cite{BannaiSloane81,BoyvalenkovDanev01}).
As an application of the uniqueness of the $D_4$ root system, we gave a new proof of the uniqueness of the $D_4$ lattice as even integral lattices of level 2 in $\R^4$.
Another application goes to a decomposition of each shell of the $D_4$ lattice in terms of orthogonal transformations of the $D_4$ root system.
We believe such decomposition would shed new light on the study of shells of a lattice.

%%%%%%%%%%%%%%%%%%%%%%%%
\section{Context of our works}

Given a finite subset $X \subset \mathbb{S}^{d-1}$, our principal problem is to determine a subset $T(X)$ of $\N$ defined by
\[ T(X)= \left\{\ell \in \N \ \middle| \  \sum_{x\in X}P(x)=0, \quad \forall P\in {\rm Harm}_\ell (\R^d)\right\},\] 
where ${\rm Harm}_\ell (\R^d)\subset \R[x_1,\ldots,x_d]$ denotes the subspace of harmonic polynomials (i.e., annihilated by the Laplacian $\sum_{j=1}^d \partial^2/\partial x_j^2$) of homogeneous degree $\ell$.
In the terminology of the design theory, the set $T(X)$ of degrees is called the {\itshape harmonic strength} of $X$.
As explained in \S1, the set $X$ is a spherical $T(X)$-design.
 
Given the coordinates of all points in $X$ and the basis of the space ${\rm Harm}_\ell (\R^d)$, the above problem will be solved.
For example, when $d=2$, %Let us illustrate two examples obtained from vertices of triangles in $\mathbb{S}^1$.
there is an explicit basis $\{f_{1,\ell},f_{2,\ell}\}$ of ${\rm Harm}_\ell (\R^2)$ such that $f_{1,\ell}(\cos\theta,\sin\theta) = \cos(\ell \theta)$ and $ f_{2,\ell}(\cos\theta,\sin\theta) = \sin(\ell \theta)$.
Let us consider the set $X_{\theta_1,\theta_2}=\{(1,0),(\cos\theta_1,\sin\theta_1),(\cos\theta_2,\sin\theta_2)\}\subset\mathbb{S}^1$ of vertices of a triangle inscribed in the unit circle.
Then, its harmonic strength can be computed by solving the equations $1+\cos(\ell \theta_1)+\cos(\ell \theta_2)=0$ and $\sin(\ell\theta_1)+\sin(\ell\theta_2)=0$.
Indeed, for an equilateral triangle (resp.~an isosceles triangle), we can show $T(X_{\frac{2\pi}{3},\frac{4\pi}{3}} ) = \{\ell \in \N \mid \ell \equiv 1,2\mod 3\}$ (resp.~$T(X_{\frac{\pi}{3},\frac{2\pi}{3}} ) = \{\ell \in \N \mid \ell \equiv 2,4\mod 6\}$).
This proves that the harmonic strength of the set of vertices of any equilateral triangle is given by $ \{\ell \in \N \mid \ell \equiv 1,2\mod 3\}$ (recall that the harmonic strength satisfies $T(\sigma(X))=T(X)$ for any orthogonal transformation $\sigma\in O(\R^d)$).
%This property would naturally lead to a classification problem of finite subsets of the unit sphere as spherical $T$-designs as in Theorem \ref{thm:main2}.

The finite subsets of $\mathbb{S}^{d-1}$ we are dealing with are the shells of a lattice in $\R^d$.
In this talk, a lattice $\Lambda$ in $\R^d$ means a full-ranked free $\Z$-module, namely, there is a basis $\{b_1,\ldots,b_d\}$ (row vectors) of $\R^d$ such that $\Lambda = \Z b_1+\cdots + \Z b_d$. 
For $m\ge0$, we define the $m$-shell $ \Lambda_m$ of $\Lambda$ by
\[ \Lambda_m = \{x\in \Lambda \mid \langle x,x\rangle=m\}.\]
Our actual problem is then to compute the harmonic strength of (non-empty) finite subsets $\frac{1}{\sqrt{m}} \Lambda_m$ of $\mathbb{S}^{d-1}$ for all $m>0$.
Unlike the triangles, this problem seems difficult; one obvious reason (from a computational obstruction) is that the cardinality is a polynomial growth $|\Lambda_m|=O(m^{d/2-1})$ as $m\rightarrow \infty$; another reason is in connection with the non-vanishing problem of the Fourier coefficients of cusp forms, as in Theorem \ref{thm:main1}.

In the next section, we will summarize some known results on the harmonic strength of some lattices and connections with the theory of modular forms.
Before we get into this, let us give a quick background on our principal problem.

%%%%%%%%%%%%%%%%%%%%%%%%
%\subsection{Background}

The following property would be one of the motivations for determining the harmonic strength (although this is a concept of spherical $t$-designs).
\begin{proposition}
If $T(X)$ contains every $\ell\in \N$ less than or equal to $t$, then we have
\begin{equation}\label{eq:spherical-design}
\frac{1}{|\mathbb{S}^{d-1}|} \int_{\mathbb{S}^{d-1}} F(\xi)d\rho(\xi) = \frac{1}{|X|} \sum_{x\in X}F(x),\quad \forall F(x)\in \R[x_1,\ldots,x_d]\ \mbox{with}\ \deg F \le t,
\end{equation}
where $\rho$ is the usual Haar measure on $\mathbb{S}^{d-1}$ and $|\mathbb{S}^{d-1}|=\int_{\mathbb{S}^{d-1}} d\rho(\xi) $. 
\end{proposition} 
\begin{proof}
The proof uses the well-known properties of harmonic polynomials.
For simplicity, we only prove the case when $t\in \N$ is even.
Recall that every polynomial $F(x) \in \R[x_1,\ldots,x_d]$ of homogeneous degree $t$ can be written in the form
\[ F(x) = P_0(x)+r_d(x)P_1(x)+r_d(x)^2P_2(x)+\cdots + r_d(x)^{\frac{t}{2}}P_{\frac{t}{2}}\]
for some $P_j(x)\in {\rm Harm}_{t-2j} (\R^d)$, where $r_d(x)=x_1^2+\cdots+x_d^2$.
Note that $P_{\frac{t}{2}}\in \R$ is a constant.
Due to the orthogonality of harmonic polynomials of different degrees, one has $\int_{\mathbb{S}^{d-1}} P_j(\xi)d\rho(\xi) = 0$ for $j=0,\ldots,\frac{t}{2}-1$.
Thus, we obitan
\[ \frac{1}{|\mathbb{S}^{d-1}|} \int_{\mathbb{S}^{d-1}} F(\xi)d\rho(\xi)  = \frac{1}{|\mathbb{S}^{d-1}|} \int_{\mathbb{S}^{d-1}} P_{\frac{t}{2}} d\rho(\xi)  = P_{\frac{t}{2}} =\frac{1}{|X|}\sum_{x\in X} P_{\frac{t}{2}} = \frac{1}{|X|} \sum_{x\in X}F(x),\]
where for the last equality, we have used the assumption that $1,2,\ldots, t\in T(X)$.
\end{proof}

A finite subset $X$ of $\mathbb{S}^{d-1}$ satisfying \eqref{eq:spherical-design} is a spherical $t$-design.
%In this concept, any finite subset on the unit sphere is labeled by $t$, called the {\itshape strength}. %, and we can study constructions and classifications of spherical $t$-designs.
As a basic fact, for any $d$ (dimension) and $t$ (strength), there exists a finite subset of $\mathbb{S}^{d-1}$ that is a spherical $t$-design, and its cardinality is bounded below by $b_{d,t}$ (called a Fisher-type bound), where 
\[ b_{d,t}= \begin{cases} \binom{d+e-1}{e}+\binom{d+e-2}{e-1} & t=2e, \\ 2\binom{d+e-1}{e} & t=2e+1.\end{cases}\]
A spherical $t$-design in $\mathbb{S}^{d-1}$ whose cardinality attains the lower bound $b_{d,t}$ is said to be {\itshape tight}.
For example, the equilateral triangle $X_{\frac{2\pi}{3},\frac{4\pi}{3}}$ is a tight spherical 2-design in $\mathbb{S}^1$, because $b_{2,2}=3$.
In general, tight designs have properties superior to non-tight designs, and do not exist for some pairs of $(d,t)$.
For more on these studies, see Bannai-Bannai \cite{BannaiBannai09}.
It should be mentioned here that a fundamental tool to give lower bounds for the cardinality is the linear programming method established by Delsarte-Goethals-Seidel \cite{DelsarteGoethalsSeidel77}.
Its machinery can be also applied to our case, so we can study tight spherical $T$-designs for a given $T\subset \N$.
For this study, we also refer to \cite{BBXYZ,BannaiOkudaTagami15,OkudaYu16,ZhuBannaiBannaiKimYu17} and references therein.
Our problem with the harmonic strength of shells is not a little influenced by these recent developments.

%%%%%%%%%%%%%%%%%%%%%%%%
\section{Modular forms and spherical designs}

%%%%%%%%%%%%%%%%%%%%%%%%
\subsection{How to use modular forms}
We outline the method, established by Venkov \cite{Venkov84}, to compute the harmonic strength of shells of lattices from the theory of modular forms.
For simplicity, we let $M_{k}(\Gamma_1(N))_\R$ be the $\R$-vector space of modular forms of weight $k$ for $\Gamma_1(N)$ whose Fourier coefficients are real numbers, where $\Gamma_1(N)=\{\gamma\in {\rm SL}_2(\Z) \mid \gamma \equiv (\begin{smallmatrix}1&\ast \\ 0&1\end{smallmatrix})\bmod N\}$ is a congruence subgroup of level $N$ of ${\rm SL}_2(\Z) $.
%Note that in the literature the notation $M_k(\Gamma_1(N))$, the $\C$-vector space of modular forms, is standard, but we know that $M_{k}(\Gamma_1(N))_\R\otimes_\R \C \cong M_k(\Gamma_1(N))$, since every modular forms is a linear combinations of newforms and the Fourier coefficients of newforms are real numbers. 

Let $\Lambda\subset \R^d$ be an even integral lattice.
Namely, $\langle x,x\rangle \in 2\Z$ holds for all $x\in \Lambda$ and $\Lambda$ is a subset of the dual lattice $\Lambda^\ast =\{y\in \R^d\mid \langle x,y\rangle \in \Z\ \mbox{for all}\ x\in \Lambda\}$.
%The minimum of all $N\in \N$ with $N \langle x,x\rangle \in 2\Z$ for all $x\in \Lambda^\ast$ is called the {\itshape level} of $\Lambda$.
For a harmonic polynomial $P\in {\rm Harm}_{\ell}(\R^d)$ and $m\in \Z_{\ge0}$, we write 
\[a_{\Lambda,P}(m)=\sum_{x\in \Lambda_{2m}}P(x)\]
 and define the {\itshape weighted theta series} $\theta_{\Lambda,P}(z)$ by 
\[ \theta_{\Lambda,P}(z) = \sum_{m\ge0} a_{\Lambda,P}(m) q^m \quad (q=e^{2\pi iz} ),\]
which is a holomorphic function on the complex upper half-plane $\{z\in \C \mid {\rm Im} \, z>0\}$.
For example, taking $P=1$ of degree 0 gives the generating series of the cardinality of each $2m$-shells of $\Lambda$.
\[\theta_{\Lambda,1}(z) = \sum_{m\ge0} | \Lambda_{2m}| q^m.\]
From the works by Hecke and Schoenberg, we see that the series $\theta_{\Lambda,P}(z)$ lies in $M_{k}(\Gamma_1(N))_\R$ (see \cite[Chap.3]{Ebeling}), where $N$ is the minimum of $N\in \N$ such that $N \langle x,x\rangle \in 2\Z$ for all $x\in \Lambda^\ast$, i.e.~the level of $\Lambda$.
We accordingly obtain the $\R$-linear map
\[ \vartheta_{\Lambda,\ell} : {\rm Harm}_{\ell}(\R^d) \longrightarrow M_{d/2+\ell}(\Gamma_1(N))_\R ,\qquad P\longmapsto \theta_{\Lambda,P}(z).\]
When $\ell\ge1$, the image 
%Denote by $\varTheta_{\Lambda,\ell}$ the $\C$-vector space spanned by all weighted theta seriess of weight $d/2+ \ell$;
\begin{equation*}
 {\rm Im}\, \vartheta_{\Lambda,\ell}= \langle \theta_{\Lambda,P}(z) \mid P\in {\rm Harm}_{\ell}(\R^d)\rangle_\R 
\end{equation*}
is a subspace of the $\R$-vector space $S_{d/2+\ell}(\Gamma_1(N))_\R$ of cusp forms in $M_{d/2+\ell}(\Gamma_1(N))_\R$.
Note that $ {\rm Im}\, \vartheta_{\Lambda,\ell}=\varnothing$ holds for $\ell$ odd, because the $m$-shell of $\Lambda$ is antipodal (i.e., $-x\in \Lambda_m$ for any $x\in \Lambda_m$).
From this, the harmonic strength $T\big(\frac{1}{\sqrt{m}} \Lambda_m\big)$ always contains all positive odd integers, so we will not consider odd $\ell$ below.

The following proposition provides an application of the theory of modular forms to design theories, which was used by Venkov in his study on even unimodular lattices \cite{Venkov84}.

\begin{proposition}\label{prop:modular-to-design}
Let $\Lambda$ be an even integral lattice in $\R^d$.
For $\ell \in 2\N$ even, there are harmonic polynomials $P_1,\ldots,P_g\in {\rm Harm}_\ell (\R^d)$ such that $\{\theta_{\Lambda,P_1}(z),\ldots,\theta_{\Lambda,P_g}(z)\}$ forms a basis of ${\rm Im}\, \vartheta_{\Lambda,\ell}$.
Then, for any positive integer $m$, we have
\[ \ell \in T\left(\frac{1}{\sqrt{2m}} \Lambda_{2m}\right) \Longleftrightarrow a_{\Lambda,P_1}(m)=\cdots =a_{\Lambda,P_g}(m)=0, \]
unless $ \Lambda_m$ is empty.
In particular, we have $\{\ell \in 2\N \mid \dim \big( {\rm Im}\, \vartheta_{\Lambda,\ell} \big)=0\} \subset T\left(\frac{1}{\sqrt{2m}} \Lambda_{2m}\right)$.
\end{proposition}
\begin{proof}
Suppose $ \ell \in T\left(\frac{1}{\sqrt{2m}} \Lambda_{2m}\right)$, which is equivalent to $a_{\Lambda,P}(m)=0$ for any $P\in {\rm Harm}_\ell(\R^d)$ of degree $\ell$.
Then it follows that $a_{\Lambda,P_j}(m)=0$ for any $j=1,\ldots,g$.
On the other hand, for any $P\in {\rm Harm}_\ell (\R^d)$, there exist $c_1,\ldots,c_g\in \R$ such that 
\[ \theta_{\Lambda,P}(z)=c_1\theta_{\Lambda,P_1}(z)+\cdots+ c_g \theta_{\Lambda,P_g}(z),\]
from which the opposite implication follows.
The ``In particular" part is clear, because if $\dim \big( {\rm Im}\, \vartheta_{\Lambda,\ell} \big)=0$, then $a_{\Lambda,P}(m)=0$ for all $P\in {\rm Harm}_\ell(\R^d)$.
\end{proof}

Using Proposition \ref{prop:modular-to-design} together with a basis of the image of the map $ \vartheta_{\Lambda,\ell}:{\rm Harm}_{\ell}(\R^d) \rightarrow S_{d/2+\ell}(\Gamma_1(N))_\R $, we can check whether $\ell $ is in the harmonic strength of $\frac{1}{\sqrt{2m}} \Lambda_{2m}$ for at least up to certain $m$ (by a computer), or for all $m$ if we are lucky.
In particular, when $\dim {\rm Im}\, \vartheta_{\Lambda,\ell}=1$, the problem whether $\ell \in   T\left(\frac{1}{\sqrt{2m}} \Lambda_{2m}\right)$ is directly connected to the vanishing problem of the $m$-th Fourier coefficients of a cusp form, so it may turn out to be a challenging problem.
%It is worth mentioning that the image sometimes coincides with the space of newforms when even integral lattices of level $N$ and discriminant\footnote{Let $B$ be a basis matrix of a lattice $\Lambda$, consisting of a basis $\{b_1,\ldots,b_d\}$. Then the discriminant of $\Lambda$ is given by $(\det B)^2$, which coincides with the determinant of the Gram matrix $Q=B{}^tB=((b_i,b_j))$ of the lattice $\Lambda$.} $D$ in $\R^d$ are unique up to orthogonal transformations (``up to orthogonal transformations" gives an indispensable class of lattices for our purpose, because, taking even integral lattices $\Lambda$ and $\Lambda'$ with the same level and discriminant such that $\Lambda'=\sigma(\Lambda)$ for some orthogonal transformation $\sigma\in O(\R^d)$, we obtain ${\rm Im}\, \vartheta_{\Lambda,\ell}={\rm Im}\, \vartheta_{\Lambda',\ell}$ for all $\ell\in 2\N$). 

In the following sections, we will discuss what is known and expected about the harmonic strength for individual important lattices.

%%%%%%%%%%%%%%%%%%%%%%%%
\subsection{The $E_8$ lattice}
A prototypical example of the use of Proposition \ref{prop:modular-to-design} is the $E_8$ lattice, the unique even integral lattice of level 1 and discriminant 1 (unimodular) in $\R^8$ (see \cite[Chap.~4, \S8]{CS} for the definition).
Its $2m$-shell is not empty for all $m\ge1$, which can be shown by the identity $\theta_{E_8,1}(z)=1+240\sum_{m\ge1}\sigma_{3}(m) q^m$ with the normalized Eisenstein series of weight 4 on ${\rm SL}_2(\Z)$, where $\sigma_k(m)=\sum_{d\mid m}d^k$ is the divisor function (see also \S\ref{sec:numerical_data} for the first several values of $|(E_8)_{2m}|$).

The $E_8$ lattice has many good properties. 
For our purpose, we recall Waldspurger's result \cite{Waldspurger79} which determines the image of the map $\vartheta_{E_8,\ell}$.

\begin{theorem}\label{thm:E_8}
Let $\ell\in 2\N$. 
For the $E_8$ lattice in $\R^8$, we have
\[ {\rm Im}\, \vartheta_{E_8,\ell} = S_{4+\ell}({\rm SL}_2(\Z))_\R.\]
\end{theorem}

%Let us illustrate a consequence of Theorem \ref{thm:E_8} for the harmonic strength of shells of the $E_8$ lattice.
From the theory of modular forms, for $\ell\in 2\N$, we have $\dim S_{4+\ell}({\rm SL}_2(\Z))_\R = \left[\frac{\ell}{4}\right] - \left[ \frac{\ell+2}{6}\right]$.
The first non-trivial cusp form appears when $\ell=8$ and it is given by $\eta(z)^{24}=\sum_{m\ge1}\tau(m)q^m$.
Now Proposition \ref{prop:modular-to-design} leads to the following corollary (note that this result can be shown without using the deep result of Waldspurger, Theorem \ref{thm:E_8}; see \cite{BannaiBannai09} and \cite[Proposition B]{Pache05}).

\begin{corollary}
For any $m\in \Z_{\ge1}$, we have
\[\{2,4,6,10\}\subset T\left(\frac{1}{\sqrt{2m}} \big(E_8\big)_{2m}\right),\]
and $8\in T\left(\frac{1}{\sqrt{2m}} \big(E_8\big)_{2m}\right)$ if and only if $\tau(m)=0$.
\end{corollary}

According to Lehmer's conjecture on the non-vanishing of Ramanujan's $\tau$-function \cite{Lehmer47} (i.e., $\tau(m)\neq0$ for all $m\in \Z_{\ge1}$), the harmonic strength of the $2m$-shell of the $E_8$ lattice would never contain eight.
Moreover, one can numerically observe up to certain $\ell$ and $m$ that $m$-th Fourier coefficients of a basis of $S_{4+\ell}({\rm SL}_2(\Z))_\R$ are not simultaneously zero (see Appendix \ref{sec:numerical_data}).
Overall, proving the equality $\{2,4,6,10\}\stackrel{?}{=} T\left(\frac{1}{\sqrt{2m}} \big(E_8\big)_{2m}\right)$ for all $m\ge1$ seems a very challenging problem.
Here, we follow the convention that all positive odd integers are excluded from the harmonic strength.

%%%%%%%%%%%%%%%%%%%%%%%%
\subsection{The $D_n$ lattice}\label{sec:Dn}
For $n\in \Z_{\ge3}$, let $D_n=\{x=(x_1,\ldots,x_n)\in \Z^n \mid x_1+\cdots+x_n\equiv 0\bmod{2}\}$ be the $D_n$ lattice.
One has $(D_n)_{2m}=\{x\in \Z^n\mid x_1^2+\cdots+x_n^2=2m\}$.
The case $n=4$ is treated in \S1, but we repeat it:
The $D_4$ lattice is known to be the unique even integral lattice of level 2 and discriminant 4 in $\R^4$.
The cardinality of the $2m$-shell of $D_4$ is given by $|(D_4)_{2m}|=12\sum_{d\mid 2m}(1-(-1)^d)d$, which is a consequence of Jacobi's four square theorem.

The image of the map $\vartheta_{D_4,\ell}$ coincides with the space of newforms of level 2 (this fact is a folklore):
\[{\rm Im}\, \vartheta_{D_4,\ell}=S_{2+\ell}^{\rm new}(\Gamma_1(2))_\R.\] 
Thus, for all $m\ge1$ (see also Appendix \ref{sec:numerical_data}), one has
\[\{2,4,10\}\subset T\left(\frac{1}{\sqrt{2m}} \big(D_4\big)_{2m}\right).\]
As is stated in Theorem \ref{thm:main1}, the harmonic strength of the $2m$-shell of the $D_4$ lattice contains six if and only if $\tau_2(m)=0$, where $\sum_{m\ge1}\tau_2(m)=\eta(z)^8\eta(2z)^8$ is the unique cusp form of weight $8$ and level 2.

As a level 2 analogue of Theorem \ref{thm:E_8}, I observed 
\[ {\rm Im}\, \vartheta_{D_8,\ell} = S_{4+\ell}(\Gamma_1(2))_\R\]
up to $\ell=10$.
Note that by Jacobi's eight-square theorem (see references in \cite{William}) the cardinality is $|(D_8)_{2m}|=r_8(2m)=16\sum_{d\mid 2m} (-1)^d d^3\neq0$ for $m\in \N$.
By the dimension formula $\dim  S_{4+\ell}(\Gamma_1(2))_\R=\left[\frac{\ell}{4}\right]$, it follows that $\{2\}\subset T\left(\frac{1}{\sqrt{2m}} \big(D_8\big)_{2m}\right)$ for all $m\ge1$.
Similarly to Theorem \ref{thm:main1}, we have $4\in T\left(\frac{1}{\sqrt{2m}} \big(D_8\big)_{2m}\right)$ if and only if $\tau_2(m)=0$.
Moreover, letting 
\[ P_4(x)=7\sum_{1\le i\le 8}x_i^4-6 \sum_{1\le i<j\le 8} x_i^2x_j^2,\]
we get $\theta_{D_8,P_4}(z)=896 \eta(z)^8\eta(2z)^8$.
Using the congruences $P_4(x)\equiv \sum_{i=1}^8 x_i^2\bmod 3$ and $P_4(x)\equiv 2(\sum_{i=1}^8 x_i^2)^2 \bmod 5$, for each odd prime $p$, we obtain
\[ 896\tau_2(p)\equiv \begin{cases} 2p | (D_8)_{2p}|\equiv  2p(1+p^3) \equiv -p(1+p)\mod{3}, \\  2(2p)^2 | (D_8)_{2p}|  \equiv p^2(1+p^3)\equiv p(p+1)\mod{5},\end{cases} \]
where we have used $|(D_8)_{2p}|=16(7+7p^3)$.
This shows that $\tau_2(p)\equiv p(p+1)\bmod \ell$ for $\ell \in \{3,5\}$ and any odd prime $p$, which was shown exactly in the same manner in \cite{HNT}.
I am expecting that $\{2\}\stackrel{?}{=} T\left(\frac{1}{\sqrt{2m}} \big(D_8\big)_{2m}\right)$ holds for all $m\ge1$ (see Appendix \ref{sec:numerical_data}).

%%%%%%%%%%%%%%%%%%%%%%%%
\subsection{The Leech lattice}
Let $\Lambda_{24}$ be the Leech lattice, the unique even unimodular lattice in $\R^{24}$ without roots (see \cite{CS,Ebeling} for the details). 
Its theta series is $\theta_{\Lambda_{24},1}(z)= 1+196560q^2+16773120q^3+\cdots\in M_{12}({\rm SL}_2(\Z))$.
Venkov \cite{Venkov84} showed that 
\begin{equation}\label{eq:leech}
{\rm Im}\, \vartheta_{\Lambda_{24},\ell} \subset \eta(z)^{48} M_{\ell -12}({\rm SL}_2(\Z))_\R,
\end{equation}
and hence, 
\[\{2,4,6,8,10,14\}\subset T\left(\frac{1}{\sqrt{2m}} \big(\Lambda_{24}\big)_{2m}\right)\] 
for all $m\ge2$.
Moreover, $12\in T\left(\frac{1}{\sqrt{2m}} \big(\Lambda_{24}\big)_{2m}\right)$ if and only if $\sum_{n=1}^{m-1}\tau(n)\tau(m-n)=0$.
In contrast to the $E_8$ and $D_4$ lattices, we can not observe if the opposite inclusion of \eqref{eq:leech} holds.
It may occur that \eqref{eq:leech} is proper for some $\ell$ and such $\ell$ may lie in the harmonic strength of $\frac{1}{\sqrt{2m}} \big(\Lambda_{24}\big)_{2m}$ for some $m\ge3$.
The above result is a part of Venkov's study of the extremal even unimodular lattices \cite{Venkov84}.

%%%%%%%%%%%%%%%%%%%%%%%%
\subsection{The Barnes-Wall lattice}
Let $BW_{16}$ be the Barnes-Wall lattice, the unique even integral lattice of level 2 and discriminant $2^8$ in $\R^{16}$ without roots (see \cite[Chap.~4, \S10]{CS} for the details). 
This is an example of extremal $2$-modular lattices.
We have $\theta_{BW_{16},1}(z)=1+4320q^2+61440q^3+\cdots$.
Similarly to the Leech lattice, one can show that
\[ {\rm Im}\, \vartheta_{BW_{16},\ell} \subset \eta(z)^{16}\eta(2z)^{16} M_{\ell -8}(\Gamma_1(2))_\R,\]
which implies
\[\{2,4,6\}\subset T\left(\frac{1}{\sqrt{2m}} \big(BW_{16}\big)_{2m}\right)\] 
for all $m\ge2$.
Again, $8\in T\left(\frac{1}{\sqrt{2m}} \big(BW_{16}\big)_{2m}\right)$ if and only if $\sum_{n=1}^{m-1}\tau_2(n)\tau_2(m-n)=0$.
It is pointed out in \cite[Remark 18.9]{Venkov01} that $10\in T\left(\frac{1}{\sqrt{2}} \big(BW_{16}\big)_{4}\right)$.
A generalization to the extremal $l$-modular lattices is studied in \cite{BachocVenkov01}.

%%%%%%%%%%%%%%%%%%%%%%%%
\subsection{Two-dimensional lattice}
So far, we have seen that it is difficult to give a complete set of the harmonic strengths for all shells.
We now show examples of two-dimensional lattices for which the harmonic strength is fully determined.

Using modular forms with complex multiplication by Gaussian integers $\Z[i]\cong \Z^2$, 
Miezaki \cite{Miezaki13} showed that 
\[ T\left(\frac{1}{\sqrt{m}} \big(\Z^2\big)_{m}\right) =\{ 2\ell\in 2\N \mid \ell\equiv 1 \bmod 2\}\]
for all $m\in \Z_{\ge1}$ with $|(\Z^2)_m|\neq0$ (note that the lattice $\Z^2$ is not even). 
With a slightly general notion, Pandey \cite{Pandey} extended this type of result to the ring of imaginary quadratic integers whose class number is one.
From his result for the Eisenstein integers $\Z[(1+\sqrt{-3})/2]$, we can deduce
\[ T\left(\frac{1}{\sqrt{2m}} \big(A_2\big)_{2m}\right) =\{ 2\ell\in 2\N \mid \ell\equiv 1,2 \bmod 3\}\]
for non-empty $2m$-shell of the $A_2$ lattice.

%%%%%%%%%%%%%%%%%%%%%%%%%%%%
\appendix

\section{Numerical results/data for $D_4,D_6,D_8,E_6,E_8$}\label{sec:numerical_data}
By examining the (simultaneous) non-vanishing of the Fourier coefficients of a basis of ${\rm Im}\, \vartheta_{\Lambda,\ell}$, we give numerical results for the harmonic strength for the cases when $\Lambda=D_4,D_6,D_8,E_6,E_8$.
Note that in these cases, corresponding modular forms are
\begin{align*}
&{\rm Im}\, \vartheta_{D_4,\ell}=S_{2+\ell}^{\rm new}(\Gamma_1(2))_\R,\\
&{\rm Im}\, \vartheta_{D_6,\ell}\subset S_{3+\ell}(\Gamma_1(4))_\R,\quad {\rm Im}\, \vartheta_{E_6,\ell}\subset S_{3+\ell}(\Gamma_1(3))_\R\\
&{\rm Im}\, \vartheta_{D_8,\ell}\stackrel{?}{=} S_{4+\ell}(\Gamma_1(2))_\R,\quad {\rm Im}\, \vartheta_{E_8,\ell}= S_{4+\ell}({\rm SL}_2(\Z))_\R.
\end{align*}
For convenience, we exhibit the dimension table of the image of the maps $\vartheta_{\Lambda,\ell}$ and the cardinality of the $2m$-shell (see \cite{Kato} for $\dim {\rm Im}\, \vartheta_{E_6,\ell}$ and \cite{CS} for the cardinality).
Note that for the cases $D_6$ and $D_8$, I have used the explicit construction of $D_n$-invariants obtained by Iwasaki \cite{Iwasaki}.
{\small
\begin{center}
\begin{tabular}{c|cccccccccccccccccccc}
$\ell$ & 2& 4&6&8&10&12&14&16&18&20&22&24\\ \hline%26&28&30
$\dim {\rm Im}\, \vartheta_{D_4,\ell}$&0&0&1& 1& 0& 2& 1& 1& 2& 2& 1& 3 \\
$\dim {\rm Im}\, \vartheta_{D_6,\ell}$&0&1&1& 2& 2& 3& 3\\%& 1& 2& 2& 1& 3 \\
$\dim {\rm Im}\, \vartheta_{D_8,\ell}$&0&1&1& 2& 2&3\\%& 2& 1& 1& 2& 2& 1& 3 \\
$\dim {\rm Im}\, \vartheta_{E_6,\ell}$&0&0&1& 1& 1& 2&2& 2& 3& 3& 3& 4 \\
$\dim {\rm Im}\, \vartheta_{E_8,\ell}$ &0&0&0& 1& 0& 1& 1& 1& 1& 2& 1&2%& 2& 2& 2\\%& 3& 2, 4, 3, 3]
\end{tabular}
\end{center}

\begin{center}
\begin{tabular}{c|cccccccccccccccccccc}
$m$ & 1& 2&3&4&5&6&7&8&9\\ \hline%26&28&30
$|(D_4)_{2m}|$ & 24& 24&96&24&144&96&192&24&312\\
$|(D_6)_{2m}|$ & 60& 252&544&1020&1560&2080&3264&4092&4380\\
$|(D_8)_{2m}|$ & 112&1136&3136&9328&14112&31808&38528&74864&84784\\
$|(E_6)_{2m}|$ & 72& 270&720&936&2160&2214&3600&4590&6552\\
$|(E_8)_{2m}|$ & 240& 2160&6720&17520&30240&60480&82560&140400&181680
\end{tabular}
\end{center}
}

To describe the results on the harmonic strength, let 
\[ T_{\le L} (X)= \left\{\ell\in 2\N  \ \middle| \  \sum_{x\in X}P(x)=0, \quad \forall P\in {\rm Harm}_\ell (\R^d), \ \ell \le L\right\}.\]

\begin{proposition}
We have checked the equalities{\small 
\begin{align*}
 T_{\le22}\left(\frac{1}{\sqrt{2m}} \big(D_4\big)_{2m}\right) &=\{2,4,10\} \quad (m\le 2\cdot 10^4),\\
 T_{\le14}\left(\frac{1}{\sqrt{2m}} \big(D_6\big)_{2m}\right) &=\{2\} \quad (m\le 2\cdot 10^4),\\
 T_{\le12}\left(\frac{1}{\sqrt{2m}} \big(D_8\big)_{2m}\right) &=\{2\} \quad (m\le 2\cdot 10^4),\\
   T_{\le10}\left(\frac{1}{\sqrt{2m}} \big(E_6\big)_{2m}\right) &=\{2,4\} \quad (m\le 2\cdot 10^4),\\
  T_{\le32}\left(\frac{1}{\sqrt{2m}} \big(E_8\big)_{2m}\right) &=\{2,4,6,10\} \quad (m\le 2\cdot 10^4).
\end{align*}}
\end{proposition}

\section*{Acknowledgments}
I would like to thank Tadashi Miyazaki for giving me the opportunity to talk at the conference.
I am also grateful to Hiroshi Nozaki and Masatake Hirao for valuable discussions, comments and suggestions.
This work is partially supported by
JSPS KAKENHI Grant Number 20K14294 and 23K03034, and the Research Institute for Mathematical Sciences,
an International Joint Usage/Research Center located in Kyoto University.

%%%%%%%%%%%%%%%%%%%%%%%%%%%%%%%%%%%%%%%%%%

\vspace{2ex}
\noindent
Koji Tasaka\\
School of Information Science \& Technology, \\
Aichi Prefectural University, \\
Nagakute 480-1198, JAPAN\\
E-mail address: tasaka@ist.aichi-pu.ac.jp

\end{document}